\documentclass[reqno, oneside, slovak]{amsart}

\usepackage[utf8]{inputenc}
\usepackage[T1]{fontenc}
\usepackage[english]{babel}
\usepackage{hyphenat}

\usepackage{float}
\usepackage[unicode=true, pdfusetitle, bookmarks=true,bookmarksnumbered=false,bookmarksopen=false,
 breaklinks=false,pdfborder={0 0 1},backref=false,colorlinks=false]{hyperref}

\usepackage{amsfonts}
\usepackage{latexsym}
\usepackage{amssymb}
\usepackage{amsmath}
\usepackage{graphicx}

\usepackage{mathrsfs}
\usepackage{subcaption} 

\usepackage{enumitem}

\usepackage{xcolor}
\usepackage{hyperref}

\numberwithin{equation}{section}

\theoremstyle{definition}
\newtheorem{df}{Definition}

\newtheorem{rem}[df]{Remark}
\theoremstyle{plain}
\newtheorem{thm}[df]{Theorem}
\newtheorem{cor}[df]{Corollary}

\newtheorem{lemma}[df]{Lemma}

\makeatletter
\@namedef{subjclassname@2020}{%
	\textup{2020} Mathematics Subject Classification}
\makeatother

\DeclareMathOperator{\id}{id}
\DeclareMathOperator{\deukl}{d_{eukl}}
\DeclareMathOperator{\dtaxi}{d_{taxi}}
\DeclareMathOperator{\metric}{d}
\DeclareMathOperator{\inter}{int}

\begin{document}

\hyphenation{ne-cessa-ri-ly fa-mi-lies con-ti-nuum Mathe-ma-tical cha-racte-rize To-po-logy con-ti-nuous mi-ni-mal equi-valen-tly non-de-gene-rate non-homeo-mor-phic stu-died Theo-rem den-drites lo-cally au-thors com-pacti-fica-tions de-fi-ni-tion de-fi-ni-tions exam-ples iso-la-ted con-di-tions Eucli-dean in-equa-li-ty con-nect-ed Fur-ther-more Frac-tals con-ti-nua Mathe-ma-tics dis-con-nect-ed met-ric homeo-mor-phisms dy-na-mical}


\title{Minimal sets on continua with a dense free interval}
\author{Michaela Mihokov\'{a}}

\address{Department of Mathematics, Faculty of Natural Sciences, Matej Bel University,
	Tajovsk\'{e}ho 40, 974 01 Bansk\'{a} Bystrica, Slovakia}

\email{michaela.mihokova@umb.sk}


\subjclass[2020]{Primary 37B05; Secondary 37B45}

\keywords{Minimal set, free interval, locally connected continuum, retract, local dendrite.}

\begin{abstract} 
	We study minimal sets on continua $X$ with a dense free interval $J$ and a locally connected remainder. This class of continua includes important spaces such as the topologist's sine curve or the Warsaw circle. In~the~case when minimal sets on the remainder are known and the remainder is connected, we obtain a full characterization of the topological structure of minimal sets. In~particular, a full characterization of minimal sets on $X$ is given in~the~case when~$X\setminus J$ is a local dendrite.
\end{abstract}

\maketitle

\thispagestyle{empty}

\section{Introduction and main results}

A \emph{(discrete) dynamical system} is an ordered pair $(X,f)$ where $X$ is a compact metrizable space and $f\colon X\to X$ is a continuous -- not necessarily invertible -- map. We say that $(X,f)$ is a \emph{minimal dynamical system} and $f\colon X\to X$ is a \emph{minimal map} if $(X,f)$ is a~dynamical system and there is no nonempty, proper, closed subset $M\subseteq X$ that is $f$-invariant (i.e., $f(M)\subseteq M$). A set $M\subseteq X$ is a~\emph{minimal set} of~$(X,f)$ if $M$ is nonempty, closed, and $f$-invariant and there is no proper subset of~$M$ having these three properties. Equivalently, $M\subseteq X$ is a \emph{minimal set} of $(X,f)$ if $M$ is nonempty, closed, and $f$-invariant and $(M, f\vert_M)$ is a minimal dynamical system. If $M\subseteq X$ is a minimal set of some dynamical system on $X$, we call $M$ a~\emph{minimal set on the space} $X$.

Birkhoff proved that any dynamical system has a minimal set. In~\cite[Theorem~B]{DSS13}, it was proved that in a compact metrizable space with a free interval $J$, every minimal set intersecting $J$ is either finite or a finite union of disjoint circles or a~(nowhere dense) \emph{cantoroid}, i.e., a compact metrizable space without isolated points in which degenerate components are dense.

A complete characterization of the topological structure of minimal sets is known only for a few classes of compact metrizable spaces listed in the next theorem (recall that a \emph{local dendrite} is a locally connected continuum containing only finitely many circles).

\begin{thm}\label{V:known_characterization}
	\leavevmode
	\renewcommand{\labelenumi}{(\roman{enumi})}
	\begin{enumerate}[label=(\roman*), ref=(\roman*)]
		\item\label{it:interval} Let X be a zero-dimensional compact metrizable space or a compact interval. Then $M\subseteq X$ is a minimal set on $X$ if and only if $M$ is either a~finite set or a~Cantor set; see, e.g., \cite[p.~92]{BC92}.
		\item\label{it:circle} Let X be a circle. Then $M\subseteq X$ is a minimal set on $X$ if and only if $M$ is either a finite set or a Cantor set or the entire $X$.
		\item\label{it:graph} Let X be a graph. Then $M\subseteq X$ is a minimal set on $X$ if and only if $M$ is either a finite set or a Cantor set or a union of~finitely many pairwise disjoint circles \cite{BHS03}.
		\item\label{it:locDend} Let X be a local dendrite. Then $M\subseteq X$ is a minimal set on $X$ if and only if $M$ is either a finite set or a cantoroid or a union of~finitely many pairwise disjoint circles \cite{BDHSS09}.
	\end{enumerate}
\end{thm}

A partial description of minimal sets is available for fibre-preserving continuous maps on graph bundles \cite{KST14} and for monotone continuous maps on local dendrites \cite{Abd15} and on regular curves \cite{DM21}.

In higher dimensional spaces, the topological structure of minimal sets is much more complicated. Finding a complete topological characterization of minimal sets of continuous selfmaps of $2$-manifolds is a very difficult problem; for a partial result, see \cite{KST08}. However, in the case of homeomorphisms, a classification of minimal sets is available on~a~$2$-torus \cite{JKP13} and extended to surfaces \cite{PX14}.

\subsection{Continua with a dense free interval}
By a \emph{free interval} in~a~space we mean an open set homeomorphic to the real interval $(0,1)$. A~\emph{ray} is a space homeomorphic to the interval $[0,\infty)$.

In this paper, we are focused on~a~special class of continua. Let $\mathcal{C}$ denote the~class of all continua $X$ that can be written in the form
\begin{equation*}
	X=L\cup J\cup R
\end{equation*}
such that
\begin{itemize}
	\item $J$ is a free interval dense in $X$,
	\item $L$, $R$ are (nowhere dense) locally connected continua disjoint with $J$ such that, for some rays $J_L$ and $J_R$ with $J=J_L\cup J_R$, $L\cup J_L$ is a compactification of $J_L$, and $R\cup J_R$ is a compactification of $J_R$.
\end{itemize}

Clearly, every such $X$ is a compactification of $J$ with the remainder $X\setminus J=L\cup R$. Throughout this paper, we also say that $L,R$ are the (left and right) remainders of~$J$. For any $X=L\cup J\cup R$ from~$\mathcal{C}$ the~remainders $L$ and $R$ are -- up to the~reverse order -- uniquely defined by $J$.

\begin{rem}\label{P:imagine}
	 We can identify the free interval $J$ with the real interval $(0,1)$ such that $L$ and $R$ are $\bigcap_{\varepsilon>0}\overline{(0,\varepsilon)}$ and $\bigcap_{\varepsilon>0}\overline{(1-\varepsilon,1)}$, respectively. The symbols $\overline{(0,\varepsilon)}$ and $\overline{(1-\varepsilon,1)}$ denote the closures of~the~sets $(0,\varepsilon)$ and $(1-\varepsilon,1)$ in $X$, respectively.
\end{rem}

The following statement and its corollary state that the class of continua $\mathcal{C}$ is very broad. In Figure~\ref{fig:Priklady_z_C}, there are four examples of such continua. In fact, $\mathcal{C}$ includes some important spaces, such~as~the~Warsaw circle (see Figure~\ref{fig:Priklady_z_C}(a)), the~topologist's sine curve (Figure~\ref{fig:Priklady_z_C}(b)), or -- when $L,R$ are singletons -- an arc or a~circle.

\begin{lemma}[\cite{MM14}]
	For each nondegenerate continuum~$P$ there are uncountably many topologically distinct (i.e., nonhomeomorphic) compactifications of a ray each with $P$ as~the remainder.
\end{lemma}

\begin{cor}
	For all locally connected continua $L^*,R^*$ where at least one of them is nondegenerate there are uncountably many topologically distinct spaces $X=L\cup J\cup R$ belonging to~$\mathcal{C}$ such that
	\begin{itemize}
		\item $L$ is homeomorphic to $L^*$,
		\item $R$ is homeomorphic to $R^*$.
	\end{itemize}
\end{cor}

\begin{figure}[h]
	\centering
	\includegraphics[width=1\textwidth]{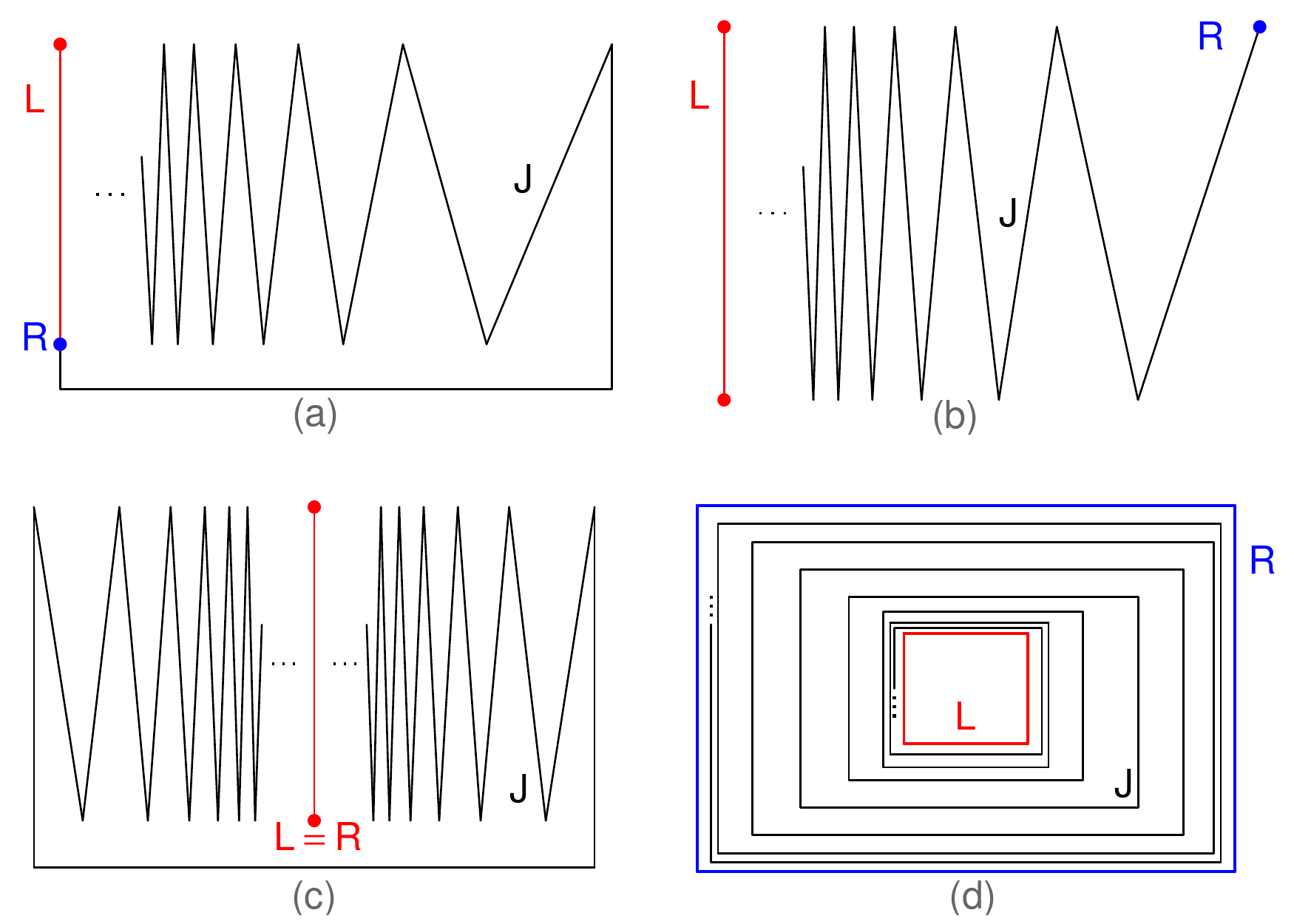}
	\caption[]{\label{fig:Priklady_z_C} Examples of continua belonging to $\mathcal{C}$.
	}
\end{figure}

Since the remainders $L,R$ of every $X\in\mathcal{C}$ are locally connected continua, they are also path connected. Thus the path components of $X$ can be described with~respect to the number of~its nondegenerate remainders.

\begin{rem}\label{P:P-C}
Up to the symmetry, for any $X=L\cup J\cup R$ from $\mathcal{C}$ exactly one of~the~following three cases is true:

\renewcommand{\labelenumi}{(\alph{enumi})}
\begin{enumerate}[label=(\alph*), ref=(\alph*)]
	\item both $L, R$ are singletons, and so $X$ is path connected,
	\item $L$ is a nondegenerate locally connected continuum and $R$ is a~singleton, and~hence:
	\begin{itemize}
		\item  if $L\cap R\neq\emptyset$ (i.e., $R\subset L$), then $X$ is path connected;
		\item if $L, R$ are disjoint, then the path components of $X$ are $L$ and $J\cup R$,
	\end{itemize}
	\item\label{it:P-C} $L, R$ are nondegenerate locally connected continua, and hence:
	\begin{itemize}
		\item if $L\cap R\neq\emptyset$, then the path components of $X$ are $J, L\cup R$;
		\item if $L, R$ are disjoint, then the path components of $X$ are $J, L, R$.
	\end{itemize}
\end{enumerate}
\end{rem}

\begin{rem}\label{P:P-Ckompakt}
	If $P$ is the path component of $X$ containing $J$, then $X\setminus P$ is compact. (Indeed, by Remark~\ref{P:P-C}, $X\setminus P$ is one of $\emptyset, L, R, L\cup R$.)
\end{rem}

\subsection{Dynamics on spaces with a dense free interval}
Since we are interested in studying minimal sets on spaces with a dense free interval, let us refer to some results on this topic. As noted above, the dynamics of continuous maps on~compact metrizable spaces containing a free interval was studied and a trichotomy for~minimal sets was proved in \cite{DSS13}. In \cite{HKO11}, the authors showed that if $X$ is a~compact metrizable space with a free interval, then every totally transitive continuous map $f\colon X\to X$ with dense periodic points is strongly mixing. Furthermore, a topological space contains a free interval if and only if it contains a free arc. Note that an \emph{arc} is a~set homeomorphic to the closed unit interval $[0,1]$, and a~\emph{free arc} in~a~topological space $X$ is an arc that is a closure of a free interval in $X$. No~locally connected continuum with a free arc admits an expansive homeomorphism \cite{Kaw88}. The notion of a free interval was also used in \cite{DHMSS19}, where the authors studied minimality for continuous actions of~abelian semigroups on compact Hausdorff spaces with such an interval.

The topologist’s sine curve, the Warsaw circle, or spaces homeomorphic to them also belong to $\mathcal{C}$. Using a compactification of the nonnegative reals whose remainder is the topologist's sine curve, some results about growths of Stone-\v{C}ech compactifications were proved in \cite{Lev77}. In \cite{SU12}, a characterization of minimal sets on~the~Warsaw circle $W$ was given. Some dynamical properties of continuous maps on~$W$ were also studied in \cite{ZPQY08} or in \cite{XYZH96}, where the authors showed that the~closure of the set of periodic points coincides with the closure of the set of recurrent points for every dynamical system on $W$ (i.e., $W$ has the periodic-recurrent property). In~\cite{Prz13}, it was proved that Milnor–Thurston homology groups for the Warsaw circle are trivial except for the zeroth homology group, which is uncountable-dimensional. For further results on the topologist’s sine curve or the Warsaw circle, see, e.g., \cite{LL19, ZZY06}.

\subsection{Main results}
The aim of this paper is to characterize a topological structure of minimal sets on continua $X=L\cup J\cup R$ belonging to $\mathcal{C}$. If the remainders $L,R$ are singletons, then the characterization is easy, since $X$ is homeomorphic to a circle or a compact interval, and~minimal sets on such continua are well-known. The main results of this paper -- involving all possibilities when at least one of the remainders is nondegenerate -- are contained in the following theorem. Before~stating it, let us introduce some notation.

\begin{df}\label{D:M(Y)}
	For a metrizable (not necessarily compact) space $Y$, the symbol $\mathcal{M}(Y)$ denotes the system of all compact sets $M\subseteq Y$ such that there exists a~continuous map $f\colon Y\to Y$ with $f(M)=M$ and $f\vert_M\colon M\to M$ being a minimal dynamical system on $M$.
\end{df}
Thus if $Y$ is compact, then $\mathcal{M}(Y)$ is the system of all minimal sets on~$Y$. In~the~case when $Y$ is noncompact, the sets $M\in\mathcal{M}(Y)$ will be also called minimal (in~fact, they are compact minimal sets of a noncompact dynamical system $(Y,f)$).

\begin{df}\label{D:M(X;Y)}
	For metrizable (not necessarily compact) disjoint spaces $X$ and $Y$, the symbol $\mathcal{M^*}(X; Y)$ denotes the system of all sets $M$ in~$\mathcal{M}(X\sqcup Y)$ such that the cardinality of $M\cap X$ is equal to the cardinality of $M\cap Y$, i.e., $\mathcal{M^*}(X; Y)$ is the system of all compact sets $M\subseteq X\sqcup Y$ such that there exists a continuous map $f\colon X\sqcup Y\to X\sqcup Y$ with $f(M)=M$, the sets $M\cap X$, $M\cap Y$ having the same cardinalities and $f\vert_M\colon M\to M$ being a minimal dynamical system.
\end{df}

\setcounter{df}{0}
\renewcommand{\thedf}{\Alph{df}}
\begin{thm}\label{V:Main}
	Let $X=L\cup J\cup R$ belong to $\mathcal{C}$.
	\begin{enumerate}
		\item\label{it:sing_nondeg_nondisj} If $R$ is a~singleton and $L$ is nondegenerate such that $R\subset L$, then
		\begin{equation*}
		\mathcal{M}(X) = \bigcup\left\{\mathcal{M}\left(L\cup A\right)\colon A\text{ is an arc, } R\subseteq A\subset R\cup J\right\}.
		\end{equation*}
		
		\item\label{it:sing_nondeg_disj} If $R$ is a~singleton and $L$ is nondegenerate such that $L\cap R=\emptyset$, then
		\begin{equation*}
		\mathcal{M}(X) = \mathcal{M}\left(L\right)\sqcup \mathcal{M}\left(J\cup R\right).
		\end{equation*}
		
		\item\label{it:nondeg_nondisj} If $L, R$ are nondegenerate such that $L\cap R\neq\emptyset$, then
		\begin{equation*}
		\mathcal{M}(X) = \mathcal{M}\left(L\cup R\right)\sqcup \mathcal{M}\left(J\right).
		\end{equation*}
		
		\item\label{it:nondeg} If $L, R$ are nondegenerate such that $L\cap R=\emptyset$, then
		\begin{enumerate}
			\item\label{it:subseteq}
			$\mathcal{M}(X) \subseteq \mathcal{M}(J)\sqcup \mathcal{M}\left(L\right)\sqcup \mathcal{M}\left(R\right)\sqcup \mathcal{M^*}\left(L; R\right),$
			\item\label{it:supseteq}
			$\mathcal{M}(X) \supseteq \mathcal{M}(J)\sqcup \mathcal{M}\left(L\right)\sqcup \mathcal{M}\left(R\right)$.
		\end{enumerate}
	\end{enumerate}
\end{thm}

Since $J$ is a free interval, minimal sets on it are well-known: $\mathcal{M}\left(J\right)$ is the system of all finite and Cantor subsets of $J$; analogously, one can describe $\mathcal{M}\left(J\cup R\right)$ in~the~case~\emph{(\ref{it:sing_nondeg_disj})}. Theorems~\ref{V:Main} and \ref{V:known_characterization} immediately yield the following corollary.

\begin{cor}
	Let $X=L\cup J\cup R$ belong to $\mathcal{C}$. Assume that $L$ and $R$ are local dendrites. Then every minimal set on $X$ is either a finite set or a cantoroid or a~union of finitely many pairwise disjoint circles. 
\end{cor}
\renewcommand{\thedf}{\arabic{df}}
\setcounter{df}{8}

In fact, we can even give a complete characterization of minimal sets on continua $X\in\mathcal{C}$ with both remainders $L$ and $R$ being local dendrites, except in the case~\emph{(\ref{it:nondeg})} when both remainders are nondegenerate and disjoint (see Corollaries~\ref{Cor:Varsava}, \ref{Cor:TSC} and \ref{Cor:nedegNedisj}). The exceptional case~\emph{(\ref{it:nondeg})} cannot be strengthened even for continua with both $L$ and $R$ being arcs. In fact, we construct continua of such a type for which the~equality in~\emph{(\ref{it:subseteq})} or in~\emph{(\ref{it:supseteq})} holds (see Sections~\ref{Sec:Pr1} and \ref{Sec:Pr2}). We do not know whether there is a continuum belonging to $\mathcal{C}$ for which neither the equality in~\emph{(\ref{it:subseteq})} nor the equality in~\emph{(\ref{it:supseteq})} holds.

The paper is organised as follows. In Section~\ref{Sec:Prelim}, we recall some necessary definitions and facts from topological dynamics. There we also prove some technical lemmas. In Section~\ref{Sec:Singletons}, we recall the characterization of minimal sets on continua with remainders being singletons. Section~\ref{Sec:singleNondeg} addresses characterization of minimal sets on continua with~exactly one nondegenerate remainder, and we prove \emph{(\ref{it:sing_nondeg_nondisj})} and \emph{(\ref{it:sing_nondeg_disj})} from Theorem~\ref{V:Main}. In Section~\ref{Sec:Nondeg}, we prove Theorem~\ref{V:Main}\emph{(\ref{it:nondeg_nondisj})} and \ref{V:Main}\emph{(\ref{it:nondeg})} and give examples of continua for which the equality in Theorem~\ref{V:Main}\emph{(\ref{it:subseteq})} or in Theorem~\ref{V:Main}\emph{(\ref{it:supseteq})} holds.

\section{Preliminaries}\label{Sec:Prelim}
We denote the set of all (positive) natural numbers by $\mathbb{N}$. The set $\mathbb{N}\cup\{0\}$ is denoted by $\mathbb{N}_0$. The set of all real numbers is denoted by $\mathbb{R}$.

The \emph{union} of sets $A,B$ is denoted by $A\cup B$; if $A,B$ are disjoint, then their union is sometimes denoted by $A\sqcup B$. The \emph{intersection} of~sets $A,B$ is denoted by~$A\cap B$. If $A$ is a \emph{subset} of $B$, we write $A\subseteq B$. The~symbol $\overline{A}$ is used for the \emph{closure} of~a~set~$A$.

Let $X,Y,Z$ be topological spaces and $f\colon X\to Y$ be a continuous map. If $A\subseteq X$, then $f\vert_A\colon A\to Y$ is a \emph{restriction} of $f$ to $A$, i.e., $f\vert_A(a)=f(a)$ for every $a\in A$. A continuous map $F\colon Z\to Y$ is called a \emph{continuous extension} of $f$ if $X\subseteq Z$ and $F(x)=f(x)$ for every $x\in X$. We say that $f$ is a \emph{homeomorphism} if $f$ is a continuous bijection and its inverse map $f^{-1}\colon Y\to X$ is also continuous. By~$\id_X\colon X\to X$ we denote the \emph{identity map} on $X$.

A topological space $X$ is a \emph{continuum} if $X$ is a nonempty, compact, connected metrizable space. A point $x\in X$ is an \emph{end point} of $X$ if for every neighborhood $U$ of $x$ there is a neighborhood $V$ of $x$ such that $V\subseteq U$ and the boundary of $V$ (i.e., the set $\overline{V}\setminus V$) is a singleton. The~set of end points of $X$ is denoted by $E(X)$.

Let $X$ be a topological space and let $x,y\in X$. A \emph{path} from $x$ to~$y$ in~$X$ is a~continuous map $f\colon [a,b]\to X$ from the compact interval $[a,b]$ into $X$ with $f(a)=x$ and $f(b)=y$. If each pair $x,y\in X$ can be joined by a path in $X$, then $X$ is said to be \emph{path connected}. The~\emph{path components} of $X$ are the maximal (with~respect to inclusion) path connected subsets of $X$.

We say that a topological space $X$ is \emph{locally connected at $x\in X$} if for every neighborhood $U$ of $x$ there exists a connected neighborhood $V\subseteq U$ of $x$. If $X$ is locally connected at each point, then we say that $X$ is \emph{locally connected}. Similarly, we say that a topological space $X$ is \emph{locally path connected at $x\in X$} if for every neighborhood $U$ of $x$ there exists a path connected neighborhood $V\subseteq U$ of $x$. If $X$ is locally path connected at each point, then we say that $X$ is \emph{locally path connected}. 

A nonempty, compact, and totally disconnected metrizable space without isolated points is called a \emph{Cantor set}.

An \emph{arc} $A$ is a set homeomorphic to the closed unit interval $[0,1]$; clearly the set $E(A)$ of end points of $A$ consists of exactly two points. An arc $A$ in a topological space $X$ is \emph{free} if $A\setminus E(A)$ is open in $X$, i.e., $A\setminus E(A)$ is a free interval in $X$.

A space $Y$ homeomorphic to $\mathbb{R}$ or to $[0,\infty)$ is called a \emph{line} or a \emph{ray}, respectively. A topological characterization of minimal sets on such spaces is well-known. (Recall that $\mathcal{M}(Y)$ was introduced in Definition~\ref{D:M(Y)}.)

\begin{lemma} \label{L:minJ}
	Let $Y$ be a line or a ray. Then
	\begin{equation*}
	\mathcal{M}(Y) = \{M\subseteq Y\colon M \text{ is either a finite set or a Cantor set}\}.
	\end{equation*}
\end{lemma}

\subsection{Retracts}

For a metrizable space $X$ and its subspace $Y\subseteq X$, a~continuous map $r\colon X\to Y$ is called a \emph{retraction} and $Y$ a \emph{retract} of~$X$ if the restriction of $r$ to $Y$ is the identity map on $Y$, i.e., $r\vert_Y=\id_Y$. We say that a compact metrizable space $A$ is an \emph{absolute retract} (in~the~class of compact metrizable spaces) if it satisfies the~following condition: whenever it is embedded as a closed subspace $A^*$ of a~compact metrizable space $S$, $A^*$ is a retract of $S$. It is a well-known fact that an arc is an~absolute retract. The next lemma provides a sufficient condition for~a~subspace of a~compact metrizable space being its retract.

\begin{lemma} \label{L:Dug}
	If $X$ is a compact metrizable space, $Y\subseteq X$ is a locally connected continuum and $\dim(X\setminus Y)\leq1$, then $Y$ is a retract of $X$.
\end{lemma}
\begin{proof}
	The subcontinuum $Y\subseteq X$ is (locally) path connected (see~\cite[Theorems 8.23 and~8.25]{Nad92}) and closed. So by~\cite[Theorem 1' on~p.~347]{Kur68}, see also~\cite[Theorem 9.1]{Dug58}, any continuous $f\colon Y\to Y$ can be continuously extended to a map $F\colon X\to Y$. For $f=\id_Y$, any such extension is a retraction of $X$ onto $Y$.
\end{proof}

The systems of all minimal sets on a metrizable space and on its retracts are closely linked.

\begin{lemma}\label{L:Mretract}
	Let $X$ be a metrizable space and $Y\subseteq X$ be a retract of~$X$. Then
	\begin{equation*}
		 \mathcal{M}(Y) = \left\{M\subseteq Y\colon M\in \mathcal{M}(X)\right\}.
	\end{equation*}
\end{lemma}
\begin{proof}
	Fix a retraction $r\colon X\to Y$.
	
	Let $M\in \mathcal{M}(Y)$ and fix a continuous map $g_Y\colon Y\to Y$ with $\left(M,g_Y\vert_M\right)$ being a~minimal dynamical system. Then $f=g_Y\circ r\colon X\to Y$ is a continuous extension of $g_Y$ over the entire $X$ such that $f\vert_M = g_Y\vert_M$ since $M\subseteq Y$. So $M$ is a minimal set of $(X,f)$, i.e., $M\in\mathcal{M}(X)$.
	
	Let $M\subseteq Y$ and $M\in \mathcal{M}(X)$, i.e., there is a continuous map $g\colon X\to X$ such that $\left(M,g\vert_M\right)$ is a minimal dynamical system; fix such $g$. Then $f_Y = r\circ g\vert_Y\colon Y\to Y$ is a continuous map. Moreover, $f_Y\vert_M = g\vert_M$, therefore $M$ is a minimal set of $\left(Y,f_Y\right)$, i.e., $M\in \mathcal{M}(Y)$. 
\end{proof}

\subsection{Path components} Path components of any $X\in\mathcal{C}$ play a significant role in~the~characterization of the topological structure of sets from $\mathcal{M}(X)$. Here we state some technical lemmas concerning them. Lemmas~\ref{L:XdoRemainderu} and \ref{L:MandJ} hold for continua from $\mathcal{C}$, but Remark~\ref{P:P-CdoP-C} and Lemma~\ref{L:permutation} hold in general.

\begin{rem}\label{P:P-CdoP-C}
	The image of a path connected set under a continuous map is path connected, since a composition of continuous maps is also continuous. Particularly, if $(X,f)$ is a dynamical system, then $f$ maps each path connected set of $X$ into~a~path component of $X$.
\end{rem}

\begin{lemma} \label{L:XdoRemainderu}
	Let $X=L\cup J\cup R$ belong to $\mathcal{C}$ and $f\colon X\to X$ be a~continuous map. If $D$ is a path component of $X$ disjoint with $J$ and~$f(J)\cap D\neq\emptyset$, then $f(X)\subseteq D$.
\end{lemma}

\begin{proof}
	Let $D$ be a path component of $X$ disjoint with $J$ such that $f(J)\cap D\neq\emptyset$. The free interval $J$ is dense in $X$, so $X=\overline{J}$ and $f(X)=f\left(\overline{J}\right)$. Moreover, $f$ is continuous and $X$ is a compact Hausdorff space, therefore $f\left(\overline{J}\right) = \overline{f(J)}$. Since $f(J)\cap D\neq\emptyset$, Remark~\ref{P:P-CdoP-C} implies $f(J)\subseteq D$.
	By Remark \ref{P:P-C}, $D$ is exactly one of $L$, $R$ or $L\cup R$; in each case, $D$ is closed. Thus $f(X)=\overline{f(J)}\subseteq \overline{D}=D$.
\end{proof}

\begin{lemma}\label{L:permutation}
	Let $(X,f)$ be a dynamical system and $M \subseteq X$ be a minimal set of it. Assume that there are path components $P_1, P_2,\dots, P_k$ ($k\in\mathbb{N}$) of $X$ such that
	\begin{equation*}
		M\subseteq \bigcup\limits_{i=1}^k P_i \quad\text{ and } \quad M\cap P_i\neq\emptyset \text{ for every } i\in\{1,2,\dots, k\}.
	\end{equation*}
	Then there is a unique permutation $\varphi$ of $\{1,2,\dots, k\}$ (i.e., a bijection of~$\{1,2,\dots, k\}$ onto itself) such that
	\begin{equation}\label{eq:L}
		f\left(P_i\right)\subseteq P_{\varphi(i)} \text{ for every } i.
	\end{equation}
	Moreover, if $M$ intersects the interior (with respect to the subspace topology on $\bigcup\limits_{i=1}^k P_i$) of some $P_i$, then $\varphi$ is a cycle.
\end{lemma}

\begin{proof}
	Put $N_k=\{1,2,\dots, k\}$. By~Remark~\ref{P:P-CdoP-C}, \eqref{eq:L}~uniquely defines a~map $\varphi\colon N_k\to N_k$. If $\varphi$ is not a bijection, then it is not surjective, and so there is $j\in N_k\setminus \varphi\left(N_k\right)$. But then $M=f(M)\subseteq \bigcup\limits_{i=1}^k f\left(P_i\right)\subseteq \bigcup\limits_{i=1}^k P_{\varphi(i)}$ does not intersect $P_j$ which contradicts an assumption of this lemma. Thus $\varphi$ is a bijection.
	
	Assume now that $M\cap \inter\left(P_i\right)\neq\emptyset$ for some $i\in N_k$, where $\inter(P_i)$ denotes the~interior of $P_i$. Fix any $j\in N_k\setminus\{i\}$ and any $x\in M\cap P_j$. Since the orbit of~$x$ is dense in $M$, there is $h\in\mathbb{N}$ such that $f^h(x)\in M\cap \inter\left(P_i\right)$. Thus, by~\eqref{eq:L}, $\varphi^h(j)=i$. Recall that every permutation can be decomposed into one or more disjoint cycles. Hence $j$ belongs to the same cycle as $i$. Since $j$ was arbitrary, $\varphi$ is a cycle.
\end{proof}

\begin{lemma} \label{L:MandJ}
	Let $X=L\cup J\cup R$ belong to $\mathcal{C}$, $f\colon X\to X$ be a continuous map and $M$ be a minimal set of $f$ with $M\cap J\neq\emptyset$. If $P$ denotes the path component of~$X$ containing $J$, then $f(P)\subseteq P$. Moreover, $M\subseteq P$.
\end{lemma}

\begin{proof}
	Denote the path components of $X$ intersecting $M$ by $P_i$ ($i=1,\dots, k$) where $1\leq k\leq 3$ and $P_1=P$. Let $\varphi$ be the permutation of~$\{1,\dots, k\}$ from Lemma~\ref{L:permutation}; since $\varphi$ intersects $J\subseteq \inter\left(P_1\right)$, $\varphi$ is a~cycle. Put $j=\varphi(1)$. If $j\neq 1$, then
	$$M=f(M)\subseteq f(X)=\overline{f(J)} = \overline{f\left(P_1\right)} \subseteq \overline{P_j}$$ 
	and $\overline{P_j}$ is disjoint with $J$ which contradicts the assumption that $M\cap J\neq\emptyset$. Thus $j=1$ and $f(P)\subseteq P$ by~\eqref{eq:L}. Further, since $\varphi$ is a cycle, $k=1$; that is, $M\subseteq P$.
\end{proof}

\section{Remainders $L, R$ are singletons}\label{Sec:Singletons}
If $X=L\cup J\cup R$ belongs to $\mathcal{C}$ and $L=R$ are singletons, then X is homeomorphic to a circle. Since minimal sets on a circle are finite sets or~Cantor sets or the entire circle (see~Theorem~\ref{V:known_characterization}\ref{it:circle}),
\begin{equation*}
	\mathcal{M}(X) = \{M\subseteq X \colon M \text{ is a finite set or a Cantor set or the entire } X\}.
\end{equation*}

If $X=L\cup J\cup R$ belongs to $\mathcal{C}$ and $L\neq R$ are distinct singletons, then $X$ is homeomorphic to a compact interval. So, by~Theorem~\ref{V:known_characterization}\ref{it:interval}, 
\begin{equation*}
\mathcal{M}(X) = \{M\subseteq X\colon M \text{ is a finite set or a Cantor set}\}.
\end{equation*}

\section{Remainder $L$ is a nondegenerate continuum and the remainder $R$ is a singleton}\label{Sec:singleNondeg}
\subsection{Remainder $R$ is a subset of the remainder $L$}

To characterize minimal sets in this case, we will use the~following lemma (\cite[Theorem~2]{SU12}, see also \cite[Theorem~D]{DHMSS19}).

\begin{lemma}\label{L:SU12}
	Let $X$ be a compact metrizable space with a free interval $J$ and $M$ be a minimal set on $X$. Then there exists a free arc $A$ in~$X$ such that $M\cap J\subseteq A\subseteq\overline{J}$. Moreover, if $M$ intersects $J$, then $M$ is contained in a closed locally connected subset of $X$.
\end{lemma}

Now we can prove the statement from Theorem~\ref{V:Main}\emph{(\ref{it:sing_nondeg_nondisj})} and its corollary for $L$ being a local dendrite.

\begin{thm}\label{V:Varsava}
	Let $X=L\cup J\cup R$ belong to $\mathcal{C}$. Assume that $R$ is a~singleton, $L$ is nondegenerate and $R\subset L$. Then
	\begin{equation*}
	\mathcal{M}(X) = \bigcup\left\{\mathcal{M}\left(L\cup A\right)\colon A\text{ is an arc, } R\subseteq A\subset R\cup J\right\}.
	\end{equation*}
\end{thm}

\begin{proof}
Fix $M\in\mathcal{M}(X)$. By Lemma~\ref{L:SU12}, there exists an arc $A\subset R\cup J$ containing $R$ such that $M\subseteq L\cup A$. Since $L\cup A$ is a locally connected continuum and $\dim \bigl(X\setminus \left(L\cup A\right) \bigr) = 1$, by Lemmas~\ref{L:Dug} and~\ref{L:Mretract} we have that $M\in\mathcal{M}(L\cup A)$.

On the other hand, assume that $M\in\mathcal{M}\left(L\cup A\right)$ for some arc $A\subset R\cup J$ containing $R$. Note that $L\cup A$ is a locally connected subcontinuum of $X$ and $\dim \bigl(X\setminus \left(L\cup A\right) \bigr) = 1$. Thus, by Lemmas~\ref{L:Dug} and~\ref{L:Mretract}, $M\in\mathcal{M}(X)$.
\end{proof}

\begin{cor}\label{Cor:Varsava}
	Let the assumptions of~Theorem~\ref{V:Varsava} hold and suppose that $L$ is a~local dendrite. Then $M\subseteq X$ is a minimal set on $X$ if and only if exactly one of~the~following conditions holds:
	\renewcommand{\labelenumi}{(\arabic{enumi})}
	\begin{enumerate}
		\item $M$ is a finite set;
		\item $M$ is a union of finitely many pairwise disjoint circles;
		\item $M$ is a cantoroid and $M\subseteq L\cup A$ for an arc $A$ such that $R\subseteq A\subset R\cup J$.
	\end{enumerate}
\end{cor}

\begin{proof}
	Trivially, for any arc $A\subseteq X$ such that $R\subseteq A\subset R\cup J$, the set $L\cup A$ is connected. Thus it is easy to see that $L\cup A$ is a local dendrite. Now, the statement follows from~Theorems~\ref{V:Varsava} and~\ref{V:known_characterization}\ref{it:locDend}.
\end{proof}

An example of a space $X=L\cup J\cup R$ belonging to $\mathcal{C}$ with $L$ being a nondegenerate continuum and $R$ being a singleton such that $R\subset L$ is the~Warsaw circle $W$ (see Figure~\ref{fig:Priklady_z_C}(a)). In this case, $L$ is a local dendrite, and therefore we obtain a complete characterization of minimal sets on $W$ from Corollary~\ref{Cor:Varsava}. Since $W$ contains no circles and every cantoroid in $W$ is a Cantor set, we get
\begin{equation*}
M\in\mathcal{M}(W)\iff M \text{ is a finite or Cantor set in } L\cup A
\end{equation*}
for an arc $A$ such that $R\subseteq A\subset R\cup J$.

\subsection{Remainders $L, R$ are disjoint}
In this subsection we prove Theorem~\ref{V:Main}\emph{(\ref{it:sing_nondeg_disj})} and its corollary for $L$ being a local dendrite.

\begin{thm}\label{V:TSC}
	Let $X=L\cup J\cup R$ belong to $\mathcal{C}$. Assume that $R$ is a~singleton, $L$ is nondegenerate and $L\cap R=\emptyset$. Then
	\begin{equation*}
	\mathcal{M}(X) = \mathcal{M}\left(L\right)\sqcup \mathcal{M}\left(J\cup R\right)
	\end{equation*}
	and $\mathcal{M}\left(J\cup R\right)$ is the system of all finite and Cantor subsets of $J\cup R$.
\end{thm}

\begin{proof}
	Recall that, by~Remark~\ref{P:P-C}, the path components of $X$ are $J\cup R$ and $L$. Take $M\in\mathcal{M}(X)$ and fix a continuous map $f\colon X\to X$ such that $\left(M, f\vert_M\right)$ is a~minimal dynamical system. If $M\cap\left(J\cup R\right) = \emptyset$, then $M\subseteq L$. Because $M\subseteq L$ is $f$-invariant, the path component $L$~is~mapped by $f$ into $L$ by~Remark~\ref{P:P-CdoP-C}. So $M$ is also a minimal set of~$f\vert_{L}\colon L\to L$, i.e., $M\in\mathcal{M}(L)$.
	
	On the other hand, if $M\cap\left(J\cup R\right) \neq \emptyset$, then there are two possibilities: either $M$ is a singleton or $M$ is not a singleton. If $M$ is a singleton, then $M$ is a fixed point of $f$, because $M$ is $f$-invariant. Then, by~Remark~\ref{P:P-CdoP-C}, the path component $J\cup R$ is mapped by $f$ into $J\cup R$. Thus $M$ is a minimal set of $f\vert_{J\cup R}\colon J\cup R\to  J\cup R$, i.e., $M\in\mathcal{M}\left(J\cup R\right)$. On~the~other hand, if $M$ is not a singleton, then $M\cap J\neq\emptyset$. Since $J\cup R$ is a path component of $X$ (and, trivially, $J\subset J\cup R$), by~Lemma~\ref{L:MandJ}, $f\left(J\cup R\right)\subseteq J\cup R$ and $M\subseteq J\cup R$. Thus $M$ is a minimal set of~$f\vert_{J\cup R}\colon J\cup R\to  J\cup R$, i.e., $M\in\mathcal{M}\left(J\cup R\right)$.
 	
	To prove the converse, assume first that $M\in\mathcal{M}\left(L\right)$. Since $L\subseteq X$ is a locally connected continuum and $\dim(X\setminus L) = 1$ (in fact, $X\setminus L=J\cup R$ is a ray), $M\in\mathcal{M}(X)$ by~Lemmas~\ref{L:Dug} and~\ref{L:Mretract}. Now assume that $M\in\mathcal{M}\left(J\cup R\right)$. The set $M$ is a compact and closed subset of~the~ray $J\cup R$, therefore there is an arc $A$ such that $M\subseteq A\subset J\cup R$. Since~the~arc $A$ is a retract of $X$, $M\in\mathcal{M}(X)$ by~Lemma~\ref{L:Mretract}.
	
	Since $J\cup R$ is a ray, the final assertion of the theorem follows from~Lemma~\ref{L:minJ}.
\end{proof}

\begin{cor}\label{Cor:TSC}
	Let the assumptions of Theorem~\ref{V:TSC} hold and suppose that $L$ is a~local dendrite. Then $M\subseteq X$ is a minimal set on $X$ if and only if exactly one of~the~following conditions holds:
	\renewcommand{\labelenumi}{(\arabic{enumi})}
	\begin{enumerate}
		\item $M\subseteq L$ and $M$ is either a finite set or a cantoroid or a union of finitely many pairwise disjoint circles;
		\item $M\subseteq J\cup R$ and $M$ is either a finite set or a Cantor set.
	\end{enumerate}
\end{cor}

\begin{proof}
	The statement follows from~Theorems~\ref{V:TSC} and~\ref{V:known_characterization}\ref{it:locDend}.
\end{proof}

An example of a space $X=L\cup J\cup R$ belonging to $\mathcal{C}$ with $L$ being a nondegenerate continuum and $R$ being a singleton such that $L\cap R=\emptyset$ is the~topologist's sine curve $TSC$ (see Figure~\ref{fig:Priklady_z_C}(b)). In this case, $L$ is a local dendrite, and therefore we obtain a~complete characterization of minimal sets on $TSC$ from Corollary~\ref{Cor:TSC}. Since $TSC$ contains no circles and every cantoroid in $TSC$ is a Cantor set, we get
\begin{equation*}
M\in\mathcal{M}(TSC)\iff
\begin{cases}
M \text{ is a finite or Cantor set in } L,\\
\qquad\qquad\text{or}\\
M \text{ is a finite or Cantor set in } J\cup R.
\end{cases}
\end{equation*}

\section{Remainders $L, R$ are nondegenerate continua}\label{Sec:Nondeg}

In this section we prove \emph{(\ref{it:nondeg_nondisj})} and \emph{(\ref{it:nondeg})} from Theorem~\ref{V:Main} and their corollaries for~$L, R$ being local dendrites. Moreover, we give examples of continua for which the equality in Theorem~\ref{V:Main}\emph{(\ref{it:subseteq})} or in Theorem~\ref{V:Main}\emph{(\ref{it:supseteq})} holds.

\subsection{The intersection of remainders $L$ and $R$ is nonempty}

\begin{thm}\label{V:nedegNedisj}
	Let $X=L\cup J\cup R$ belong to $\mathcal{C}$. Assume that $L, R$ are nondegenerate and $L\cap R\neq\emptyset$. Then
	\begin{equation*}
	\mathcal{M}(X) = \mathcal{M}\left(L\cup R\right)\sqcup \mathcal{M}\left(J\right)
	\end{equation*}
	and $\mathcal{M}(J)$ is the system of all finite and Cantor subsets of $J$.
\end{thm}

\begin{proof}
Recall that, by~Remark~\ref{P:P-C}, the path components of $X$ are $J$ and~$L\cup R$. Let $M\in\mathcal{M}(X)$. Fix a continuous map $f\colon X\to X$ such that $\left(M,f\vert_M\right)$ is a minimal dynamical system. If $M\cap J=\emptyset$, then $M\subseteq L\cup R$. Because $M\subseteq L\cup R$ is $f$-invariant, $L\cup R$ is mapped by $f$ into $L\cup R$ by~Remark~\ref{P:P-CdoP-C}. So $M$ is also a minimal set of $f\vert_{L\cup R}\colon L\cup R\to L\cup R$, i.e., $M\in\mathcal{M}\left(L\cup R\right)$.

Now assume that $M\in\mathcal{M}(X)$ is such that $M\cap J\neq\emptyset$. The free interval $J$ is a~path component of $X$, therefore, by~Lemma~\ref{L:MandJ}, $f(J)\subseteq J$ and $M\subseteq J$. Thus $M$ is a minimal set of $f\vert_J\colon J\to J$, i.e., $M\in\mathcal{M}(J)$.

On the other hand, assume that $M\in\mathcal{M}\left(L\cup R\right)$. Since both $L,R\subseteq X$ are locally connected continua and $L\cap R\neq\emptyset$, also $L\cup R$ is a locally connected continuum. Moreover, $\dim\left(X\setminus\bigl(L\cup R\bigr)\right)=1$ (in fact, $X\setminus\bigl(L\cup R\bigr)=J$ is an interval), and therefore, by~Lemmas~\ref{L:Dug} and~\ref{L:Mretract}, $M\in\mathcal{M}(X)$.

Finally, suppose that $M\in\mathcal{M}\left(J\right)$. Since $M$ is a compact subset of~a~free interval $J$, there exists an arc $A\subset J$ with $M\subseteq A$. Since $A$ is a~retract of $X$, $M\in \mathcal{M}(X)$ by~Lemma~\ref{L:Mretract}.

Since $J$ is a free interval, the final assertion of the theorem follows from~Lemma~\ref{L:minJ}.
\end{proof}

\begin{cor}\label{Cor:nedegNedisj}
	Let the assumptions of~Theorem~\ref{V:nedegNedisj} hold and suppose that $L, R$ are local dendrites. Then $M\subseteq X$ is a minimal set on $X$ if and only if exactly one of~the~following conditions holds:
	\renewcommand{\labelenumi}{(\arabic{enumi})}
	\begin{enumerate}
		\item $M\subseteq L\cup R$ and $M$ is either a finite set or a cantoroid or a~union of finitely many pairwise disjoint circles;
		\item $M\subseteq J$ and $M$ is either a finite set or a Cantor set.
	\end{enumerate}
\end{cor}

\begin{proof}
	The statement follows from~Theorems~\ref{V:nedegNedisj} and~\ref{V:known_characterization}\ref{it:locDend}.
\end{proof}

An example of a space $X=L\cup J\cup R$ belonging to $\mathcal{C}$ with $L, R$ being nondegenerate continua such that $L\cap R\neq\emptyset$ is in Figure~\ref{fig:Priklady_z_C}(c). In this case, $L$ and $R$ are local dendrites, and therefore we obtain a complete characterization of minimal sets on~such~$X$ from Corollary~\ref{Cor:nedegNedisj}.

\subsection{Remainders $L, R$ are disjoint}

Recall that $\mathcal{M^*}(X;Y)$ was introduced in~Definition~\ref{D:M(X;Y)}.

\begin{thm}\label{V:3}
	Let $X=L\cup J\cup R$ belong to $\mathcal{C}$. Assume that $L, R$ are nondegenerate and $L\cap R=\emptyset$. Then
	\begin{enumerate}
	\item\label{it:V3sub}
	$\mathcal{M}(X) \subseteq \mathcal{M}(J)\sqcup \mathcal{M}\left(L\right)\sqcup \mathcal{M}\left(R\right)\sqcup \mathcal{M^*}\left(L; R\right),$
	\item\label{it:V3sup}
	$\mathcal{M}(X) \supseteq \mathcal{M}(J)\sqcup \mathcal{M}\left(L\right)\sqcup \mathcal{M}\left(R\right)$
	\end{enumerate}
	and $\mathcal{M}(J)$ is the system of all finite and Cantor subsets of $J$.
\end{thm}

\begin{proof}
	\emph{(1)} Assume that $M\in\mathcal{M}(X)$. Fix a continuous map $f\colon X\to X$ such that $\left(M,f\vert_M\right)$ is a minimal dynamical system. By~Remark~\ref{P:P-C}\ref{it:P-C}, the path components of $X$ are $J, L$ and $R$. If $M\cap J\neq \emptyset$, then, by~Lemma~\ref{L:MandJ}, $f(J)\subseteq J$ and $M\subseteq J$. Thus $M$ is a minimal set of~$f\vert_J\colon J\to J$, i.e., $M\in\mathcal{M}(J)$.
	
	Now, assume that $M\cap J=\emptyset$, and so $M\subseteq L\sqcup R$. If also $M\cap R=\emptyset$, then $M\subseteq L$. Since $M\subseteq L$ is $f$-invariant, Remark~\ref{P:P-CdoP-C} implies that the~path component $L$ is mapped by $f$ into $L$. Therefore $M$ is a~minimal set of $f\vert_{L}\colon L\to L$, i.e., $M\in\mathcal{M}\left(L\right)$.
	
	If $M\cap J=\emptyset$ and also $M\cap L=\emptyset$, then $M\subseteq R$. In this case we can use the~same arguments as above and conclude that $M\in\mathcal{M}\left(R\right)$.
	
	Finally, let $M\cap J=\emptyset$, $M\cap L\neq\emptyset$ and $M\cap R\neq\emptyset$. By~Lemma~\ref{L:permutation}, either $f(L)\subseteq L$ and $f(R)\subseteq R$, or $f(L)\subseteq R$ and $f(R)\subseteq L$. The former case contradicts minimality of $M$ (in fact, both the subsets $M\cap L$ and $M\cap R$ of $M$ would be nonempty, closed and $f$-invariant), thus the latter case is true. Since $f(M)=M$, we have $f(M\cap L) = M\cap R$ and vice versa, and so the cardinality of $M\cap L$ is equal to the cardinality of $M\cap R$. Moreover, since $M\subseteq L\sqcup R$, $f\left(L\right)\subseteq R$ and $f\left(R\right)\subseteq L$, the set $M$ is minimal for $f\vert_{L\sqcup R}\colon L\sqcup R \to L\sqcup R$. Thus $M\in \mathcal{M^*}\left(L; R\right)$.
	
	\emph{(2)} Assume that $M\in\mathcal{M}(J)$. The set $M$ is compact in the free interval $J$, therefore there is an arc $A\subset J$ containing $M$. Since $A$ is a~retract of both $J$ and $X$, $M\in \mathcal{M}(X)$ by~Lemma~\ref{L:Mretract}.
	
	Finally, we prove that $M\in\mathcal{M}\left(L\right)$ implies $M\in\mathcal{M}(X)$; proving that $M\in\mathcal{M}\left(R\right)$ implies $M\in\mathcal{M}(X)$ is analogous, so we are not going to do that. Let $M\in\mathcal{M}\left(L\right)$ and fix a continuous map $g_L\colon L\to L$ such that $\left(M,g_L\vert_M\right)$ is a minimal dynamical system. By~Remark~\ref{P:imagine}, we can identify the free interval $J$ with the real interval $(0,1)$ such that $L$ and $R$ are $\bigcap\limits_{\varepsilon>0}\overline{(0,\varepsilon)}$ and $\bigcap\limits_{\varepsilon>0}\overline{(1-\varepsilon,1)}$, respectively. Denote $J_0=(0,1/2]$ and $X_0=L\sqcup J_0$. Using this notation, $X_0$ is the compactification of the ray $J_0$ with the remainder $L$. Since $L\subseteq X_0$ is a~locally connected continuum and $\dim(X_0\setminus L)=1$ (in fact, $X_0\setminus L=J_0$ is a~ray), by~Lemma~\ref{L:Dug} there is a retraction $r_L\colon X_0\to L$. Define
	\begin{equation*}
	f_L(x)=
	\begin{cases}
	g_L\circ r_L(x) & \text{if } x\in X_0,\\
	g_L\circ r_L(1/2) & \text{if } x\in X\setminus X_0.
	\end{cases}
	\end{equation*}
	Obviously, $f_L\colon X\to L$ is a continuous extension of $g_L$ over the entire $X$. Moreover, $M$ is a minimal set of $\left(X,f_L\right)$. Thus $M\in\mathcal{M}(X)$.
	
	Since $J$ is a free interval, the final assertion of the theorem follows from~Lemma~\ref{L:minJ}.
\end{proof}

\begin{cor}\label{Cor:nedegDisj}
Let the assumptions of~Theorem~\ref{V:3} hold and suppose that $L, R$ are local dendrites. Then:
\renewcommand{\labelenumii}{(\alph{enumii})}
\begin{enumerate}
	\item If $M\subseteq X$ is a minimal set on $X$, then exactly one of the following conditions holds:
	\begin{enumerate}
		\item $M\subseteq J$ and $M$ is either a finite set or a Cantor set;
		\item either $M\subseteq L$ or $M\subseteq R$, and $M$ is either a finite set or a cantoroid or a union of finitely many pairwise disjoint circles;
		\item $M\subseteq L\sqcup R$, the sets $M\cap L$, $M\cap R$ have the same cardinalities and $M$ is either a finite set or a cantoroid or a union of~finitely many pairwise disjoint circles.
	\end{enumerate}
	\item If any of the conditions (a), (b) holds, then $M$ is a minimal set on $X$.
\end{enumerate}
\end{cor}

\begin{proof}
	The statements follow from Theorems~\ref{V:3} and~\ref{V:known_characterization}\ref{it:locDend}.
\end{proof}

Examples of spaces $X=L\cup J\cup R$ belonging to $\mathcal{C}$ with $L, R$ being nondegenerate, disjoint continua are in Figure~\ref{fig:Priklady_z_C}(d) and also in the remaining two subsections.

\subsection{A continuum with equality in Theorem~\ref{V:3}(\emph{\ref{it:V3sub}})}\label{Sec:Pr1}
In this subsection we give an example of a continuum $X$ for which the equality in~Theorem~\ref{V:3}(\emph{\ref{it:V3sub}}) (or, equivalently, in~Theorem~\ref{V:Main}\emph{(\ref{it:subseteq})}) holds. By $(x,y)$ and $[x,y]$ we denote an open interval and a closed interval with end points $x,y$, respectively. We use the~symbol $x\times y$ to denote a point in $\mathbb{R}^2$ with the first coordinate $x$ and the second coordinate $y$. The symbol $\deukl$ denotes the Euclidean distance; the symbol $ab$ denotes a line segment with end points $a,b\in\mathbb{R}^2$ where $a\neq b$. Recall that $\mathbb{N}_0$ denotes the set $\mathbb{N}\cup\{0\}$ where the symbol $\mathbb{N}$ denotes the set of all (positive) natural numbers.

\smallskip
\emph{Step~1. Construction of the continuum $X$.} We construct a space $X$ (see Figure~\ref{fig:2TSC}) which is a subcontinuum of $[0,2]\times[0,1]$.
\smallskip 

\begin{figure}[h]
	\centering
	\includegraphics[width=1\textwidth]{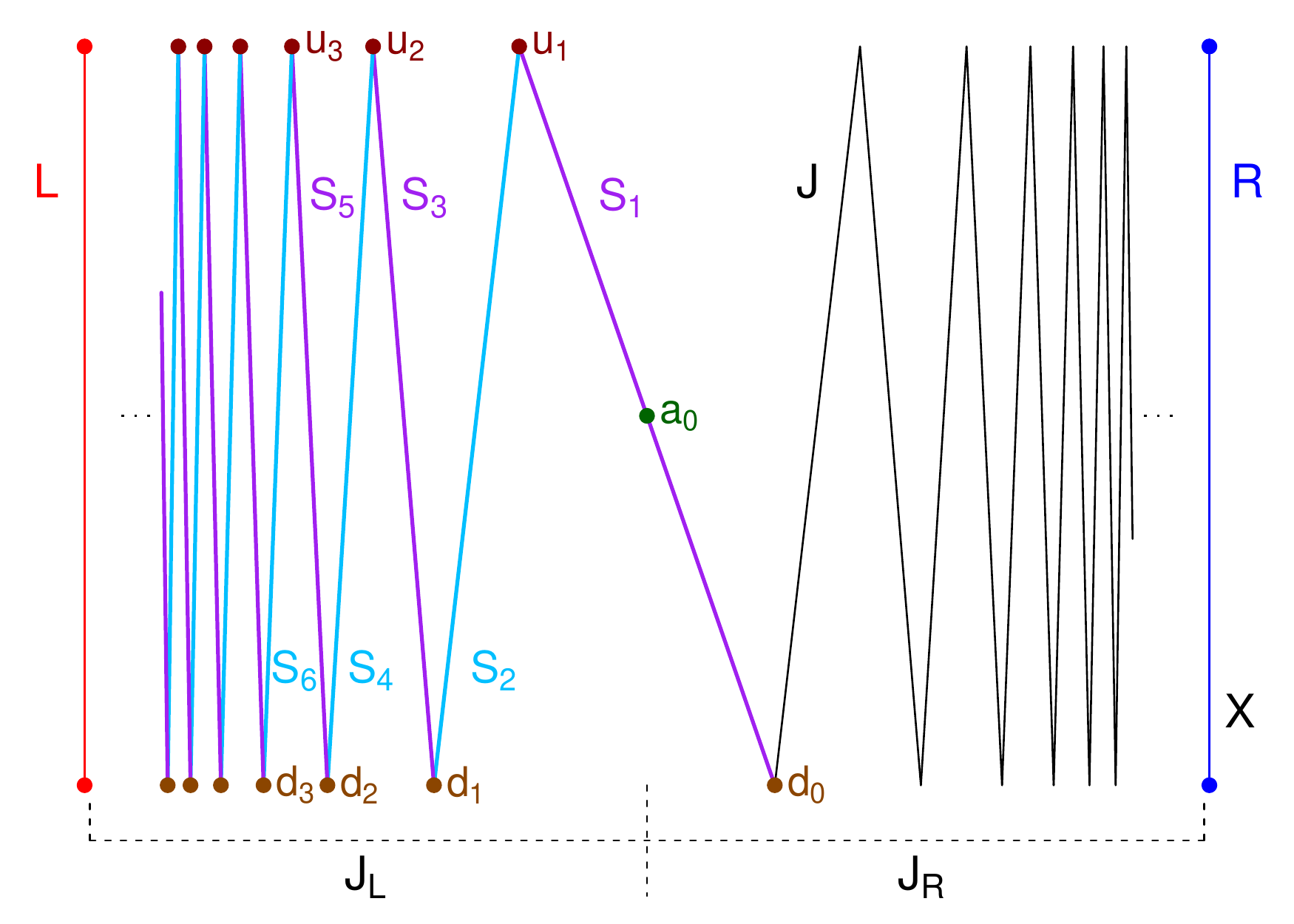}
	\caption[]{\label{fig:2TSC} The continuum $X$.
	}
\end{figure}

Define a continuous map $\varphi\colon (0,2)\to [0,1]$ such that:
\renewcommand{\labelenumi}{(\roman{enumi})}
\begin{itemize}
	\item $\varphi(1)=1/2$,
	\item $\varphi\left(\frac{1}{2n}\right)=1$ and $\varphi\left(\frac{1}{2n+1}\right)=0$ for every $n\in\mathbb{N}$,
	\item $\varphi$ is linear on $\left[\frac{1}{n+1},\frac{1}{n}\right]$ for every $n\in\mathbb{N}$,
	\item $\varphi(2-x)=1-\varphi(x)$ for every $x\in (1,2)$.
\end{itemize}

Let $J$ denote the graph of $\varphi$ and define $\Phi\colon (0,2)\to J$ such that $\Phi(x) = x\times\varphi(x)$ for every $x\in (0,2)$, i.e., $J=\left\{\Phi(x)\colon x\in(0,2)\right\}$. We define:
\begin{itemize}
	\item the ordering on $J$ such that $\Phi(x)<\Phi(y)$ if $x<y$ (where $x,y\in J$),
	\item the so-called taxi distance $\dtaxi$ on $J$ such that $\dtaxi\left(\Phi(x),\Phi(y)\right)$ is the (Euclidean) length of the set $\left\{\Phi(z)\colon z\in[x,y]\right\}$ (where $x,y\in J$); obviously $\deukl\vert_{J\times J}\leq\dtaxi$.
\end{itemize}

Let $u_n=\Phi\left(\frac{1}{2n}\right)$ and $d_n=\Phi\left(\frac{1}{2n+1}\right)$ for every $n\in\mathbb{N}$. Denote $d_0=\Phi\left(\frac32\right)$ and $a_0=\Phi(1)$, i.e., $a_0$ is the midpoint of $d_0u_1$. We use the~symbols $S_{2n-1}$ and $S_{2n}$ ($n\in\mathbb{N}$) to denote the line segments $d_{n-1}u_n$ and $u_{n}d_n$, respectively. The union of~$a_0u_1$ and $\bigcup_{n=2}^\infty S_n$ is denoted by~$J_L$.

Now, put $L=\bigl\{0\times y\colon 0\leq y\leq1 \bigr\}$. Obviously, $L\sqcup J_L$ is the closure of $J_L$ in~$[0,2]\times[0,1]$. The set $L\sqcup J_L$ is homeomorphic to the~topologist's sine curve.

We denote the set $a_0\cup \left(J\setminus J_L\right)$ by $J_R$ . Put $R=\bigl\{2\times y\colon 0\leq y\leq1 \bigr\}$; then $R\sqcup J_R$ is the closure of $J_R$ in~$[0,2]\times[0,1]$. Moreover, $R\sqcup J_R$ is homeomorphic to~the~topologist's sine curve.

Finally, put $X=L\cup J\cup R$. Trivially, $X$ belongs to $\mathcal{C}$ with $L, R$ being nondegenerate and disjoint. It is easy to see that $X$ is homeomorphic to the so-called double topologist's sine curve.

We define a map $\pi_{R}\colon J\to R$ to be the natural projection from~$J$ onto~the~remainder $R$, i.e., the $\pi_R$-image of $x\times\varphi(x)$ is $2\times\varphi(x)$ for~every $x\in(0,2)$. Analogously, $\pi_{L}\colon J\to L$ is the natural projection from $J$ onto~the~remainder $L$, i.e., the $\pi_L$-image of $x\times\varphi(x)$ is $0\times\varphi(x)$  for~every $x\in (0,2)$.

For any interval $K\subseteq (0,2)$ put $X_K=\left\{\Phi(x)\colon x\in K\right\}$. Further, for $a=\Phi(s)  < \Phi(t)=b$ from $J$, $[a,b]$ denotes the set $\left\{\Phi(x)\colon x\in [s,t]\right\}$.

\smallskip
\emph{Step~2. Definition of $f_L^R\colon L\sqcup R\to L\sqcup R$.}
Fix $M\in\mathcal{M^*}\left(L; R\right)$ and take a~continuous map $f_M\colon L\sqcup R\to L\sqcup R$ such that $\left(M,f_M\vert_M\right)$ is a minimal dynamical system. In this step, we construct a continuous and surjective map $f_L^R\colon L\sqcup R\to L\sqcup R$ that extends $f_M\vert_M$.
\smallskip

The space $X$ has three path components: $J, L, R$ (see~Remark~\ref{P:P-C}\ref{it:P-C}). Since the~distance between $L$ and $R$ is positive, Lemma~\ref{L:permutation} (applied to the space $L\sqcup R$ and the map $f_M$) yields that $f_M\left(L\right)\subseteq R$ and $f_M\left(R\right)\subseteq L$. In particular, $f_M\left(M\cap L\right)\subseteq R$ and $f_M\left(M\cap R\right)\subseteq L$. By~the~Tietze extension theorem there exist continuous maps $f_M^L\colon L\to R$ and $f_M^R\colon R\to L$ that extend $f_M\vert_{M\cap L}$ and $f_M\vert_{M\cap R}$, respectively. Since $L,R$ are arcs, without loss of generality, we can suppose that these maps are surjective. The map $f_L^{R}\colon L\sqcup R\to L\sqcup R$ defined by
\begin{equation*}
	f_L^{R}(x)=
	\begin{cases}
	f_M^L(x) & \text{if } x\in L,\\
	f_M^R(x) & \text{if } x\in R,
	\end{cases}
\end{equation*}
is well defined and continuous. Moreover, it is surjective and $f_L^{R}$ extends $f_M\vert_M$.

\smallskip
\emph{Step~3. Construction of points $a_n, b_n$ in $J$.}
\smallskip

Put $b_0=a_0$, and $i_0=1$. Inductively construct points $a_n, b_n$ ($n\geq 1$) in $J$ such that, for every $n\in\mathbb{N}$, 
\begin{enumerate}[label=(\roman*), ref=(\roman*)]
	\item\label{it:Pr_ind1} $a_n\in X_{(1,2)}, \quad a_{n-1}<a_n, \quad\dtaxi\left(a_{n-1},a_{n}\right)=1/n$;
	\item\label{it:Pr_ind2} $b_n^*=f_L^R\left(a_n^*\right)$ where $b_n^*=\pi_{L}\left(b_n\right)$ and $a_n^*=\pi_{R}\left(a_n\right)$;
	\item\label{it:Pr_ind3} $b_n\in X_{(0, 3/2]}\cap S_{i_n}$ for some $i_n\in\left\{i_{n-1}, i_{n-1}+1\right\}$;
	\item\label{it:Pr_ind4} $i_n=i_{n-1}+1$ if and only if $\dtaxi\left(b_{n-1},c_{n}\right)<1/i_{n-1}$ where $c_n=\pi_L^{-1}\left(b_n^*\right)\cap S_{i_{n-1}+1}$.
\end{enumerate}
A situation when $\dtaxi\left(b_{n-1},c_{n}\right)\geq1/i_{n-1}$ is illustrated in~Figure~\ref{fig:2TSCab}.
\begin{figure}[h]
	\centering
	\includegraphics[width=1\textwidth]{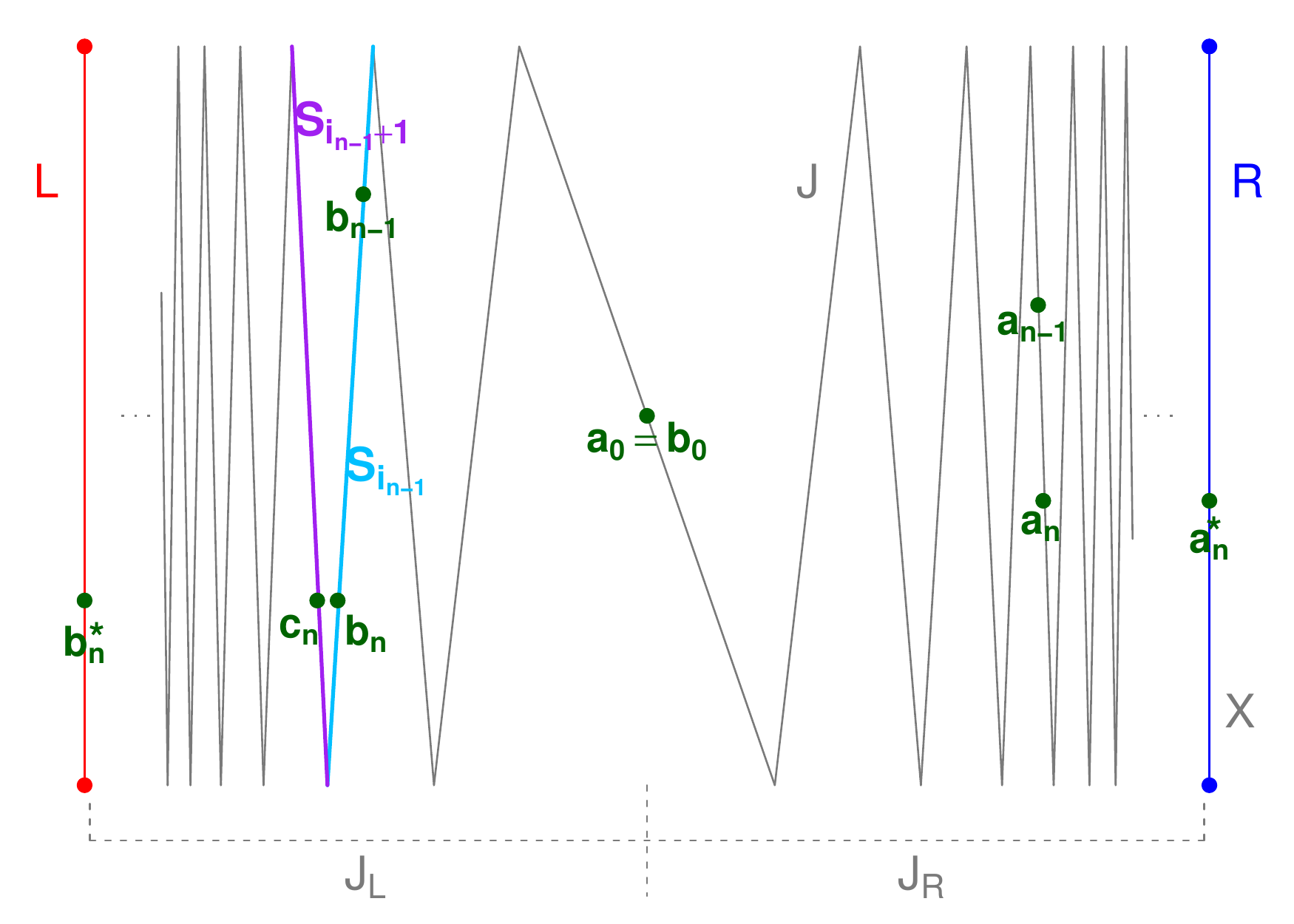}
	\caption[]{\label{fig:2TSCab} An illustration of the specific points.
	}
\end{figure}
Clearly, $\left(a_n\right)_{n=1}^\infty$ is uniquely defined by~\ref{it:Pr_ind1} and~$a_0$. Further, realize that, for any $b^*\in L$, $\pi_L^{-1}\left(b^*\right)$ intersects every $S_i$ in a single point. Thus~\ref{it:Pr_ind2}, \ref{it:Pr_ind3} and~\ref{it:Pr_ind4} uniquely define $\left(b_n\right)_{n=1}^\infty$. Note also that $b_n=c_n$ if the~condition from~\ref{it:Pr_ind4} is satisfied. Continuity of $f_L^R$ and $\pi_R$, together with~\ref{it:Pr_ind1} and~\ref{it:Pr_ind2}, yield

\begin{equation}\label{eq:Pr1.db_star_0}
\lim\limits_{n\to\infty}\deukl\left(b_n^*,b_{n-1}^*\right)=0.
\end{equation}

\smallskip
\emph{Step~4. Proof of $i_n\to\infty$ for $n\to\infty$.}
\smallskip

Suppose, on the contrary, that $i_n\not\to\infty$. Then, by~\ref{it:Pr_ind3}, there are $i,n_0\in\mathbb{N}_0$ such that

\begin{equation}\label{eq:Pr1.i_ni}
	i_n=i \text{ for every }n\geq n_0.
\end{equation}
Put $z=S_{i}\cap S_{i+1}$. Then, by~\ref{it:Pr_ind1} and surjectivity of $f_L^R$ and $\pi_R$, there is an increasing sequence $\left(n_k\right)_{k=1}^\infty$ of integers such that

\begin{equation}\label{eq:Pr1.b_star_z}
	\lim\limits_{n\to\infty} b^*_{n_k} = \pi_L(z).
\end{equation}
Now, by~\eqref{eq:Pr1.i_ni} and linearity of $\pi_L$ on the straight line segment $S_i$, 

\begin{equation}\label{eq:Pr1.b_z}
\lim\limits_{n\to\infty} b_{n_k} = z.
\end{equation}
As in~\ref{it:Pr_ind4} put $c_{n_k+1}=\pi_L^{-1}\left(b_{n_k+1}^*\right)\cap S_{i+1}$ for every $k$. Then $c_{n_k+1}\to z$ by~\eqref{eq:Pr1.db_star_0} and~\eqref{eq:Pr1.b_star_z}. So, analogously as above,
\begin{equation*}
\lim\limits_{n\to\infty} c_{n_k+1} = z.
\end{equation*}
This equation together with~\eqref{eq:Pr1.b_z} and~\ref{it:Pr_ind4} imply that, for every sufficiently large $k$, $i_{n_k+1}=i_{n_k}+1$, a contradiction with~\eqref{eq:Pr1.i_ni}. This proves that $\lim\limits_{n\to\infty}i_n=+\infty$.

\smallskip
\emph{Step~5. Definition of a continuous surjective map $f_J\colon J\to J$.}
\smallskip

Define $f_{J_R}\colon J_R\to J_L\cup S_1$ such that the following is true for every $n\in\mathbb{N}_0$:
\begin{itemize}
	\item $f_{J_R}\left(a_n\right)=b_n$;
	\item $f_{J_R}\vert_{\left[a_n,a_{n+1}\right]}$ is a piecewise linear surjection of $\left[a_n,a_{n+1}\right]$ onto $\left[b_n,b_{n+1}\right]$ with~constant slope.
\end{itemize}
Clearly, $f_{J_R}$ is uniquely defined and continuous. Since
\begin{equation}\label{eq:Pr1.a_b_0}
\lim\limits_{n\to\infty}\deukl\left(a_n,a_{n}^*\right)=0\quad \text{ and }\quad \lim\limits_{n\to\infty}\deukl\left(b_n,b_{n}^*\right)=0
\end{equation}
by Step~4 and~\ref{it:Pr_ind2}, $f_{J_R}\left(J_R\right) \supseteq J_L$.

Analogously construct $f_{J_L}\colon J_L\to J_R\cup S_1$ (with analogous choice of $\left(a_n^\prime\right)_{n=1}^\infty$, $\left(b_n^\prime\right)_{n=1}^\infty$) and define $f_J\colon J\to J$ by
\begin{equation*}
	f_J(x)=
	\begin{cases}
	f_{J_R}(x) & \text{if } x\in J_R,\\
	f_{J_L}(x) & \text{if } x\in J_L.
	\end{cases}
\end{equation*}
Clearly, $f_J\colon J\to J$ is continuous and surjective.

\smallskip
\emph{Step~6. Proof of uniform continuity of $f_J$.}
\smallskip

We first show that
\begin{equation}\label{eq:Pr1.b_0}
	\lim\limits_{n\to\infty}\dtaxi\left(b_n,b_{n-1}\right)=0.
\end{equation}
To this end, put $N_0=\left\{n\in\mathbb{N}\colon i_n=i_{n-1}\right\}$ and $N_1=\left\{n\in\mathbb{N}\colon i_n=i_{n-1}+1\right\}$. By~\ref{it:Pr_ind3}, $\mathbb{N}=N_0\sqcup N_1$. The set $N_1$ is infinite by Step~4 and, by~\ref{it:Pr_ind4},
\begin{equation*}
	\dtaxi\left(b_n,b_{n-1}\right)<1/i_{n-1} \text{ for every } n\in N_1,
\end{equation*}
hence
\begin{equation}\label{eq:Pr1.N10}
\lim\limits_{\substack{n\to\infty\\ n\in N_1}}\dtaxi\left(b_n,b_{n-1}\right)=0.
\end{equation}
Clearly, by~\eqref{eq:Pr1.db_star_0}, also $N_0$ is infinite. Take any $n\in N_0$ and put $i=i_n=i_{n-1}$. Since $\pi_L\vert_{S_i}\colon S_i\to L$ is linear with absolute value of the slope equal to $1/|S_i|$ (where $|S_i|$ denotes the Euclidean length of $S_i$),
\begin{equation*}
	\dtaxi\left(b_n,b_{n-1}\right) = |S_i|\cdot \deukl\left(b_n^*,b_{n-1}^*\right) \leq |S_1|\cdot \deukl\left(b_n^*,b_{n-1}^*\right).
\end{equation*}
Thus, by~\eqref{eq:Pr1.db_star_0},
\begin{equation}\label{eq:Pr1.N00}
\lim\limits_{\substack{n\to\infty\\ n\in N_0}}\dtaxi\left(b_n,b_{n-1}\right)=0.
\end{equation}
Since $\mathbb{N} = N_0\sqcup N_1$, \eqref{eq:Pr1.N10} and~\eqref{eq:Pr1.N00} yield \eqref{eq:Pr1.b_0}.

Now the proof of uniform continuity of $f_J$ is straightforward. To this end, take any sequence $\left(x_n\right)_{n=1}^\infty$ in $J$ converging to $x\in X$. If $x\in J$, then clearly $f_J\left(x_n\right)\to f_J(x)$ for $n\to\infty$. Otherwise, $x\in L\sqcup R$; we may assume that $x\in R$ and that every $x_n\in J_R$. Then for every $n$ there is $k_n\geq 0$ such that
\begin{equation*}
	x_n\in {\left[a_{k_n},a_{k_{n}+1}\right]}.
\end{equation*}
Then $k_n\to\infty$ since $x_n\to x\in R$, and so $a_{k_n}\to x$ and $\dtaxi\left(a_{k_n}, a_{k_n+1}\right)\to 0$ by~\ref{it:Pr_ind1}. Continuity of $f_L^R$ and $\pi_R$ now yields that, for $b^*=f_L^R(x)\in L$,
\begin{equation*}
	\lim\limits_{n\to\infty}f_L^R\left(\pi_R\left(x_n\right)\right) = \lim\limits_{n\to\infty}f_L^R\left(\pi_R\left(a_{k_n}\right)\right) = b^*.
\end{equation*}
This, together with~\ref{it:Pr_ind2}, \ref{it:Pr_ind3} and~Step~4, give $b_{k_n}\to b^*$. Now, since $f_J\left(x_n\right)\in f_J\left(\left[a_{k_n},a_{k_{n}+1}\right]\right) = \left[b_{k_n},b_{k_{n}+1}\right]$ and $\dtaxi\left(b_{k_n},b_{k_n+1}\right)\to0$ by~\eqref{eq:Pr1.b_0}, we have that $f_J\left(x_n\right)\to b^*=f_L^R(x)$.

We have proved that $f_J$ is continuously extensible to the entire compactum $X$, hence $f_J$ is uniformly continuous.

\smallskip
\emph{Step~7. Definition of a continuous surjective map $f\colon X\to X$.}
\smallskip

Define $f\colon X\to X$ by 
\begin{equation*}
f(x)=
\begin{cases}
f_L^R(x) & \text{if } x\in L\sqcup R,\\
f_J(x) & \text{if } x\in J.
\end{cases}
\end{equation*}
The map $f$ is surjective since $f_L^R\colon L\sqcup R\to L\sqcup R$ and $f_J\colon J\to J$ are such. Since $f_L^R$ and $f_J$ are continuous, to prove continuity of~$f$ it suffices to show that $f\left(x_n\right)\to f(x)$ for every $x\in L\sqcup R$ and every sequence $\left(x_n\right)_{n=1}^\infty$ in $J$ converging to $x$. To this end, fix any such $x$ and any $\left(x_n\right)_{n=1}^\infty$; we may assume that $x\in R$. By~\ref{it:Pr_ind1}, there is a~subsequence $\left(a_{k_n}\right)_{n=1}^\infty$ of $\left(a_k\right)_{k=1}^\infty$ such that $a_{k_n}\to x$. Then $b_{k_n}^* = f_L^R\left(\pi_R\left(a_{k_n}\right)\right) \to f_L^R\left(x\right) = f(x)$. Further, by~\eqref{eq:Pr1.a_b_0}, $\deukl\left(b_{k_n},b_{k_n}^*\right)\to 0$. Hence $f\left(a_{k_n}\right) = b_{k_n}\to f(x)$. Since $\deukl\left(x_{n},a_{k_n}\right)\to 0$ and $f_J$ is uniformly continuous by~Step~6, we have that $\deukl\left(f\left(x_{n}\right),f\left(a_{k_n}\right)\right) =\deukl\left(f_J\left(x_{n}\right),f_J\left(a_{k_n}\right)\right)\to 0$. Thus $f\left(x_n\right)\to f(x)$.

\smallskip
\emph{Step~8. Proof that $\mathcal{M}(X) = \mathcal{M}(J)\sqcup \mathcal{M}\left(L\right)\sqcup \mathcal{M}\left(R\right)\sqcup \mathcal{M^*}\left(L; R\right)$.}
\smallskip

In~Steps~2-7 we have proved that every $M\in\mathcal{M^*}\left(L; R\right)$ belongs to~$\mathcal{M}(X)$. The~rest follows from~Theorem~\ref{V:3}.

\subsection{A continuum with equality in~Theorem~\ref{V:3}(\emph{\ref{it:V3sup}})}\label{Sec:Pr2}
In this subsection we show the existence of a continuum for which the equality in~Theorem~\ref{V:3}(\emph{\ref{it:V3sup}}) (or, equivalently, in~Theorem~\ref{V:Main}\emph{(\ref{it:supseteq})}) holds. We use the following result:

\begin{lemma}[\cite{Awa93}, Theorem 4.7]\label{L:Awa93}
	There exists an uncountable collection $\mathcal{S}$ of~cardinality $\mathfrak{c}$ of compactifications of the ray with the arc as~remainder, no member of which maps continuously onto any other.
\end{lemma}

Take distinct $X_L, X_R\in\mathcal{S}$ and write $$X_L=J_L\sqcup L,\qquad X_R=J_R\sqcup R$$ where $J_L$ ($J_R$) is a ray dense in $X_L$ ($X_R$) and $L$ ($R$) is a nowhere dense arc. We may assume that $X_L\cap X_R=\{x_0\}$ where $x_0$ is the unique end point of both $X_L$ and $X_R$. Identify $J_L$ with $(0,1/2]$ and $J_R$ with $[1/2,1)$ such that $L= \bigcap_{\varepsilon>0}\overline{(0,\varepsilon)}$ and $R =\bigcap_{\varepsilon>0}\overline{(1-\varepsilon,1)}$, where the symbols $\overline{(0,\varepsilon)}$ and $\overline{(1-\varepsilon,1)}$ denote the closures of the sets $(0,\varepsilon)$ and $(1-\varepsilon,1)$ in $X_L$ and $X_R$, respectively. Put $X=X_L\cup X_R$. Obviously, $X$ is from $\mathcal{C}$ with the dense free interval $J_L\cup J_R$ and nondegenerate disjoint remainders $L$ and $R$. With respect to the previous sections, denote $J_L\cup J_R$ by $J$. Let $\metric$ be a (compatible) metric on $X$.

We need to show that there is no minimal set $M$ on $X$ such that $M\in\mathcal{M}^*(L;R)$. Suppose, on the contrary, that there is such $M$. Fix a continuous map $f\colon X\to X$ with $\left(M,f\vert_M\right)$ being a minimal dynamical system. By~Remark~\ref{P:P-C}\ref{it:P-C}, $X$ has three path components: $L, R$ and $J=J_L\cup J_R$. The cardinality of $M\cap L$ is equal to~the~cardinality of $M\cap R$, because $M\in\mathcal{M}^*(L;R)$. By~Lemma~\ref{L:permutation}, $f(L)\subseteq R$ and~$f(R)\subseteq L$. Thus, by~Lemma~\ref{L:XdoRemainderu}, $f(J)\subseteq J$.

Take points $m_L\in M\cap L, m_R\in M\cap R$ and fix $\varepsilon>0$. Since $f$ is continuous, there is $\delta_1>0$ such that for a point $x_1\in J$ with~$\metric\left(x_1, m_L\right)<\delta_1$ the inequality $\metric\left(f(x_1), f(m_L)\right)<\varepsilon$ holds. Such point $x_1$ exists since $J$ is dense in $X$. This fact and continuity of $f$ together imply that there is also $\delta_2>0$ and a point $x_2\in J$ with~$\metric\left(x_2, m_R\right)<\delta_2$ and $\metric\left(f(x_2), f(m_R)\right)<\varepsilon$. The free interval $J$ is connected, therefore its image under the continuous map $f$ is also connected. Moreover, since~$\varepsilon>0$ is arbitrarily small, the distance between $f(J)$ and $R$ is equal to $0$ and the same is the distance between $f(J)$ and $L$. So, $f(J)=J$, and because $J$ is dense in $X$, $f(X)=f\left(\overline{J}\right)=\overline{J}=X$. Therefore $f$ is surjective. Now, $f(L)\subseteq R$ and $f(R)\subseteq L$ imply $f(L)=R$ and $f(R)=L$. Since the inclusions $L\subsetneq X_L\subsetneq J\cup L$ hold, $R\subsetneq f\left(X_L\right)\subseteq J\cup R$. Moreover, $X_L$ is a continuum, therefore its $f$-image is compact and connected, and so $f\left(X_L\right)=[c,1)\cup R$ for~some $c\in(0,1)$. It is easy to~see that $f\left(X_L\right)$ is homeomorphic to $X_R$. But this is a contradiction to the choice of $X_L$ and $X_R$.

\medskip
\noindent\emph{Acknowledgements.}
The author is very obliged to Vladim\'{i}r \v{S}pitalsk\'{y} for careful reading of the manuscript and many suggestions for improving the paper. Furthermore, the author thanks \v{L}ubom\'{i}r Snoha for useful discussions.
This work was supported by VEGA grant 1/0158/20.


\end{document}